\documentclass[12pt]{amsart}
\usepackage[alphabetic]{amsrefs}
\usepackage[english]{babel}
\usepackage{graphics}
\usepackage{amsfonts}
\usepackage{amsmath}
\usepackage{amssymb}
\usepackage{mathrsfs}
\usepackage{enumerate}
\usepackage{relsize,exscale}
\usepackage{cite}
\usepackage{esint}
\usepackage{MnSymbol}
\usepackage{epsfig}
\usepackage{amscd}

\usepackage{graphicx}
\usepackage{newlfont}
\usepackage{latexsym}
\usepackage{amsthm}
\usepackage{color}

\newtheorem{theorem}{Theorem}[section]
\newtheorem{lemma}[theorem]{Lemma}
\newtheorem{corollary}[theorem]{Corollary}
\newtheorem{proposition}[theorem]{Proposition}

\theoremstyle{remark}

\numberwithin{equation}{section}

\theoremstyle{definition}

\renewcommand{\epsilon}{\varepsilon}
\newcommand{\e}{\epsilon}

\newcommand{\R}{\mathbb{R}}

\newcommand{\N}{\mathbb{N}}

\newcommand{\POI}{\mathbf{P}}
\newcommand{\EEEE}{\mathbf{E}}

\makeindex
\begin{document}

\title[Minimizing cones for fractional capillarity problems]{Minimizing cones \\ for fractional capillarity problems}

\author[S. Dipierro]{Serena Dipierro}
\address{Serena Dipierro \textsuperscript{1}}
\author[F. Maggi]{Francesco Maggi}
\address{Francesco Maggi \textsuperscript{2}}
\author[E. Valdinoci]{Enrico Valdinoci}
\address{Enrico Valdinoci \textsuperscript{1}}

\address{{\textsuperscript{1}
University of Western Australia,
Department of Mathematics and Statistics,
35 Stirling Highway,
Crawley, Perth,
WA6009, Australia}}

\email{serena.dipierro@uwa.edu.au, enrico.valdinoci@uwa.edu.au}

\address{{\textsuperscript{2}
University of Texas at Austin,
Department of Mathematics,
2515 Speedway,
Stop C1200, Austin TX 78712-1202, USA}
}

\email{maggi@math.utexas.edu}

\thanks{The first author is supported by
the Australian Research Council DECRA DE180100957
``PDEs, free boundaries and applications''. The second author is supported by
the NSF Grant DMS 2000034. The third author is supported by
the Australian Laureate Fellowship
FL190100081
``Minimal surfaces, free boundaries and partial differential equations''.
The first and third authors are members of INdAM and AustMS}


\begin{abstract} We consider a fractional version of Gau{\ss} capillarity energy. A suitable extension problem is introduced to derive a boundary monotonicity formula for local minimizers of this fractional capillarity energy. As a consequence, blow-up limits of local minimizers are shown to subsequentially converge to minimizing cones. Finally, we show that in the planar case there is only one possible fractional minimizing cone, the one determined by the fractional version of Young's law.
\end{abstract}

\maketitle

\section{Introduction}

In this article we consider local minimizers in the fractional capillarity model introduced in~\cite{MR3717439}, analyze their blow-up limits at boundary points, show, by means of a new monotonicity formula, that these blow-up limits are cones, and give a complete characterization of such cones in the planar case.

\medskip

In the classical capillarity model of Gau{\ss}, see \cite{MR816345}, one studies equilibrium configurations of liquid droplets $E$ in a container $\omega\subset\R^n$, $n\ge2$, by looking at (volume-constrained) local minimizers of the (dimensionally re-normalized) surface tension energy 
\[
{\mathcal{H}}^{n-1}(\omega\cap\partial E)+ \sigma {\mathcal{H}}^{n-1}(\partial\omega\cap\partial E)\,,
\]
where $\sigma\in(-1,1)$ is the (constant) {\it relative adhesion coefficient} determined by the physical properties of the liquid and of the walls of the container. In the model introduced in~\cite{MR3717439}, see \eqref{CAP} below, the liquid-air surface energy term ${\mathcal{H}}^{n-1}(\omega\cap\partial E)$ is replaced by the nonlocal interaction between points $x\in E$ and in $y\in\omega\setminus E$; while the liquid-solid surface energy term ${\mathcal{H}}^{n-1}(\partial\omega\cap\partial E)$ 
is replaced by the nonlocal interaction between points $x\in E$ and $y\not\in\omega$. These nonlocal interactions are measured by the singular fractional kernel $|x-y|^{s-n}$, $s\in(0,1)$: as $s\to 1^-$, they are increasingly concentrated, respectively, at points $x$ and $y$ near $\omega\cap\partial E$ and $\partial\omega\cap\partial E$. For this reason, {\it the fractional capillarity model provides a nonlocal approximation of the Gau{\ss} capillarity model in the limit $s\to 1^-$}.

This happens also at the level of the classical equilibrium conditions expressed by the constancy of the mean curvature of $\omega\cap\partial E$ and by the contact angle condition between the liquid-air interface and the walls of the container, valid along $\partial\omega\cap\overline{\omega\cap\partial E}$, and known as {\bf Young's law}. The validity of a {\bf fractional Young's law} (see \eqref{ANGO} below) for sufficiently regular local minimizers of the fractional capillarity energy has been proved in~\cite{MR3717439}, while its precise asymptotics in the limits $s\to 1^-$ and $s\to 0^+$ have been presented in~\cite{MR3707346}. The existence of minimizers in the fractional capillarity model is also addressed in 
\cite{MR3717439}. It is an open problem to understand if these minimizers are regular up to the boundary of the container $\omega$, and thus to confirm the validity of the fractional Young's law in a pointwise sense. In this paper we take two important steps in what is a general and well-established strategy for attacking similar questions in geometric variational problems. 

Our first result (given in Corollary~\ref{HOMCON}) is that blow-up limits of local minimizers subsequentially converge to cones (which, in turn, are also local minimizers). This result relies on a new monotonicity
formula for the fractional capillarity energy (see Theorem~\ref{MONOTONICITY}) and on
an equivalence result with a suitable ``capillarity adaptation'' of the Caffarelli-Silvestre extension problem (given in Proposition~\ref{EXTENSION}).

Our second result (stated in  Theorem~\ref{2CONE})
is a classification theorem for fractional minimizing cones in the half-plane:
more precisely, we will show that the only possible fractional minimizing cones in ambient dimension~$2$ are angular sectors satisfying the fractional version of Young's law.

While the first result about the blow-up limits
(as well as the extension theorem and the monotonicity
formula used in its proof) is valid in any dimension,
the second result about classification of cones
is only proved in dimension~$2$, due to suitable energy
estimates that would not be valid in higher dimensions. 
It is an interesting open problem, which is also open
for interior singularities for arbitrary values of $s\in(0,1)$, to understand if similar rigidity
results for minimizing cones are valid in higher dimensions. The other main open problem  is that of obtaining a boundary regularity criterion comparable
to the one available in the interior \cite{MR2675483}, and analogous to the ones
developed in the classical case to validate Young's law, see \cite{MR3317808,MR3630122} and the references
therein.

\medskip

The precise mathematical setting in which we work is the following.
Given~$s\in(0,1)$ and two disjoint sets~$A$, $B\subseteq\R^n$, we define the {\bf fractional interaction} between $A$ and $B$ as
$$ {\mathcal{I}}_s(A,B):=\iint_{A\times B}\frac{dx\,dy}{|x-y|^{n+s}}.$$
Then, given~$E\subseteq\omega\subseteq\R^n$
and~$\sigma\in(-1,1)$, we define the {\bf fractional
capillarity energy of $E$ in $\omega$} as
\begin{equation}\label{CAP} {\mathcal{C}}_{s,\sigma}(E,\omega):={\mathcal{I}}_s(E,E^c\omega)+\sigma
{\mathcal{I}}_s(E,\omega^c).\end{equation}
Here above and in the rest of this paper, we use the superscript~``$c$''
to denote the complementary set in~$\R^n$. Also, given two sets~$A$, $B\subseteq\R^n$
we use the short notation~$AB:=A\cap B$ (in this way, the notation~$E^c\omega$
is short for~$(\R^n\setminus E)\cap\omega$). Furthermore, the Lebesgue
measure of a set~$F\subseteq\R^n$ will be denoted by~$|F|$.

We consider the half-space
$$H:=\{x=(x_1,\dots,x_n)\in\R^n{\mbox{ s.t. }}x_n>0\},$$
and, given~$R>0$, we denote
by~$B_R\subset\R^n$ the $n$-dimensional Euclidean ball of radius~$R$ centered at the origin.
In this article, we are interested in local minimizers of the fractional capillarity energy in~$H$.
Briefly, we say that~$E\subseteq H$ is a {\bf{local minimizer in~$H$}} if, for every~$R>0$,
we have that~${\mathcal{I}}_s(E B_R, E^c B_R)<+\infty$ and
\begin{equation}\label{MINIMOBLOW}
\begin{split}& {\mathcal{I}}_s(E B_R, E^c H)+{\mathcal{I}}_s(E B_R^c,E^c B_R H)+\sigma
{\mathcal{I}}_s(E B_R,H^c)
\\ \le\;& {\mathcal{I}}_s(F B_R, F^c H)+{\mathcal{I}}_s(F B_R^c,F^c B_R H)+\sigma
{\mathcal{I}}_s(F B_R,H^c),\end{split}\end{equation}
for every~$F\subseteq H$ such that~$F\setminus B_R=E\setminus B_R$.
In particular, blow-up limits of minimizers in the fractional capillarity
problem in bounded domains with smooth boundary are local minimizers in~$H$,
see~\cite[Theorem A.2]{MR3717439}.

In order to exploit extension methods (see e.g.~\cite{MR2354493}), for any~$(x,t)\in\R^{n+1}_+:=
\R^n\times(0,+\infty)$, it is convenient to
introduce the fractional Poisson kernel
$$ \POI_s(x,t) := C_{n,s}\,\frac{t^s}{(|x|^2+t^2)^{\frac{n+s}{2}}},$$
where~$C_{n,s}>0$ is a normalizing constant (which, from now on, will be omitted)
such that
$$ \int_{\R^n} \POI_s(x,t)\,dx=1,\qquad{\mbox{for all }}t>0.$$
Given~$u\in L^\infty(\R^n)$, we also denote the {\bf $s$-extension of~$u$}
by
$$ \EEEE_u(x,t):=\int_{\R^n} u(y)\,\POI_s(x-y,t)\,dy,\qquad{\mbox{ for all }}
(x,t)\in\R^{n+1}_+\,.$$
The relevance of this notion of $s$-extension for our problem lies in the fact that the property of $E$ being a local minimizer in $H$ for the fractional capillarity energy ${\mathcal{C}}_{s,\sigma}$ is equivalent to the property of $\EEEE_u$ being a local minimizer of a Dirichlet-type energy ${\mathcal{F}}_{s,\sigma}$ that we are now going to introduce. Indeed, let~$X=(x,t)\in\R^{n+1}_+$.
As customary, given~$E\subseteq\R^n$, we denote by~$\chi_E:\R^n\to\{0,1\}$
the characteristic function of~$E$.
If~$u=\chi_E$, we also write~$\EEEE_E:=\EEEE_{\chi_E}$.
In addition, given~$\Omega\subseteq\R^{n+1}$ with~$\omega:=\Omega\cap\{t=0\}$
and~$\Omega^+:=\Omega\cap\R^{n+1}_+$,
and~$U:\R^{n+1}\to\R$ with~$u(x):=U(x,0)$, we define the energy
\begin{equation}\label{FSSIG} {\mathcal{F}}_{s,\sigma}(U,\Omega):=\int_{\Omega^+} t^{1-s} |\nabla U(X)|^2
\,dX+
(\sigma-1)\iint_{\omega\times H^c} \frac{u(x)}{|x-y|^{n+s}}\,dx\,dy.
\end{equation}
Given~$K\subseteq \R^{n+1}$ and~$\eta>0$, we also set
\begin{equation}\label{1.3BIS} K_\eta:=\{ X\in\R^{n+1} {\mbox{ s.t. }} {\mathrm{dist}}(X,K)<\eta\}.\end{equation}
Then, we have the following extension result:

\begin{proposition}\label{EXTENSION}
Let~$E\subseteq H$ be such that~${\mathcal{I}}_s(E B_R, E^c B_R)<+\infty$
for every~$R>0$.
The following statements are equivalent:
\begin{enumerate}
\item[(i).] $E$ is a local minimizer in $H$.
\item[(ii).] For all~$R>0$
and all bounded, Lipschitz domains~$\Omega\subset\R^{n+1}$
with
\begin{equation}\label{TRACCIA}
\Omega\cap \{t=0\}= B_R ,\end{equation} we have that
\begin{equation}\label{TRACCIA2} {\mathcal{F}}_{s,\sigma}(\EEEE_E,\Omega)\le {\mathcal{F}}_{s,\sigma}(U,\Omega)\end{equation}
for all~$U:\R^{n+1}_+\to\R$ such that~$U(x,0)=\chi_F(x)$ for all~$x\in\R^n$,
for some~$F\subseteq H$, with~$F B_{R-\eta}^c=EB_{R-\eta}^c$,
and~$U(X)=\EEEE_E(X)$
for all~$X\in (\partial\Omega)_\eta\cap\R^{n+1}_+$, for some~$\eta\in(0,R)$.
\end{enumerate}
\end{proposition}

The previous result can be seen as the natural counterpart,
in the setting of fractional capillarity problems.
of several extension theorems for the fractional Laplacian,
for fractional minimal surfaces and, more generally, for nonlocal free
boundary problems, see e.g. in~\cite{MR104296, MR0247668, MR2354493, MR2675483, MR2754080, MR3712006}.
Among the many applications of the powerful tool provided
by extension results, is the possibility of obtaining convenient monotonicity
formulae: actually, to the best of our knowledge,
all the monotonicity formulae
involving nonlocal operators rely on identifying appropriate local extension problems methods.

In the setting considered in this paper, we will exploit
Proposition~\ref{EXTENSION} to obtain a monotonicity formula that
we now describe in detail.
We denote by~${\mathcal{B}}_R\subset\R^{n+1}$ the $(n+1)$-dimensional Euclidean ball of radius~$R$.
For~$E\subseteq\omega$ and~$r>0$, we define
\begin{equation*} \Phi_E (r):= {r^{s-n}}
{\mathcal{F}}_{s,\sigma}(\EEEE_E,{\mathcal{B}}_r).\end{equation*}
We observe that the above function is scale invariant, in the sense that
\begin{equation}\label{SCALE}
\Phi_E (r)= \Phi_{E_{r/\rho}} (\rho),
\end{equation}
where
\begin{equation}\label{BUS}
E_r:=\frac{ E }{r}=\left\{ \frac{x}{r},\;\, x\in E\right\}.\end{equation}
In this setting, we have the following monotonicity formula:

\begin{theorem}\label{MONOTONICITY}
Assume that~$E\subseteq H$ is a local minimizer for the fractional capillarity energy
in~$H$.
Then, the function~$(0,+\infty)\ni r\mapsto \Phi_E(r)$ is monotone nondecreasing.

More precisely, for every~$r>0$ we have that
\begin{equation}\label{MONI}
\Phi_E'(r)\ge r^{s-n}
\int_{ (\partial{\mathcal{B}}_{r})\cap\{t>0\}}
t^{1-s} |\nabla_\nu \EEEE_{ E }(X)|^2\,d{\mathcal{H}}^n_{X}.
\end{equation}
Furthermore, we have that~$\Phi_E$ is constant if and only if~$E$ is a cone,
i.e.~$\tau E=E$ for all~$\tau>0$.
\end{theorem}

As a consequence of Theorem~\ref{MONOTONICITY}, we have that suitable blow-up
limits of local minimizers of the fractional capillarity problem are cones:

\begin{corollary}\label{HOMCON}
Let~$\omega\subset\R^n$ be a bounded open set with~$C^1$-boundary.
Let~$E\subseteq\omega$ be a minimizer
of the capillarity functional in~\eqref{CAP} among sets of prescribed volume contained in $\omega$.

Assume that~$0\in\overline{\omega\cap(\partial E)}$. Then for every vanishing sequence $r_j$ there exists (a not relabeled subsequence and) a set~$E_0\subset\R^n$, such that, in the notation of~\eqref{BUS},
we have that~$\chi_{E_{r_j} }\to\chi_{E_0}$ in~$L^1_{\mathrm loc}(\R^n)$.
In addition, $E_0$ is a cone.
\end{corollary}

The existence of the minimizers in Corollary~\ref{HOMCON}
(and, in fact, of a more general class of minimizers)
is warranted by Proposition~1.1 in~\cite{MR3717439}.
As a matter of fact, Corollary~\ref{HOMCON} is also valid
for the ``almost minimizers'', as introduced in
Definition~1.5 of~\cite{MR3717439}, with the same proof that we present here.
\medskip

In the setting of Corollary~\ref{HOMCON}, it is natural to consider
locally minimizing cones in~$H$ (i.e., sets that are locally minimizing in~$H$
and that possess a conical structure).
Interestingly, in dimension~2, we can completely characterize
locally minimizing cones in~$H$, according to the following result:

\begin{theorem}\label{2CONE}
Let~$n=2$ and~$E$ be a locally minimizing cone in~$H=\{x_2>0\}$.
Then, $E$
is made of only one component and, up to a rigid motion, we have that
$$ E=\{ x=(x_1,x_2)\in H {\mbox{ s.t. }} x_1> x_2\cos\vartheta\},$$
with~$\vartheta\in(0,\pi)$ implicitly defined by the formula
\begin{equation}\label{ANGO}\begin{split}&
1+\sigma=\frac{(\sin\vartheta)^s\,\;M(\vartheta,s)}{M(\pi/2,s)},\\
{\mbox{where }}\;&M(\vartheta,s):=
2\iint_{(0,\vartheta)\times(0,+\infty)}\frac{r}{(r^2+2r\cos t+1)^{\frac{2+s}2}}\,dt\,dr.
\end{split}
\end{equation}
\end{theorem}

Notice that~\eqref{ANGO} expresses the fractional Young's law mentioned earlier in this introduction, which, in the limit as $s\to 1^-$ converges to the contact angle prescription given
by the classical Young's Law. For a detailed asymptotic description
of this feature, see~\cite{MR3707346}.

To prove Theorem~\ref{2CONE},
we utilize a ``translation method'' that was
introduced in~\cite{MR3090533} to prove
the regularity of fractional minimizing surfaces in the plane.
In our context, however, the cone is going to have a singularity
at the origin, hence the notion of ``regularity'' has to be weaken
to a suitable notion of ``monotonicity'', taking inspiration
by some work in~\cite{MR3035063}.

The rest of this paper is devoted to the proof of the results
that we have presented above. More specifically, Section~\ref{CAPS}
contains some preliminary observations relating the nonlocal surface
tension energy introduced in~\cite{MR3717439}
and the nonlocal perimeter functional introduced in~\cite{MR2675483}. Then,
the proof of Proposition~\ref{EXTENSION} will be given in Section~\ref{SEXTA}
and the one of
Theorem~\ref{MONOTONICITY} in Section~\ref{B09j}.
Section~\ref{COPROOF} contains the proof of
Corollary~\ref{HOMCON} and
Section~\ref{coniin2} the one of
Theorem~\ref{2CONE}.

\section{Capillarity and fractional perimeters}\label{CAPS}

In this section, we point out some useful relations
between the capillarity functional in~\eqref{CAP} and
other fractional energies of geometric type.
First of all,
we observe that the energy functional in~\eqref{CAP} can be related to
the fractional perimeter introduced in~\cite{MR2675483}.
Indeed,
writing, for any given~$F$, $\omega\subseteq\R^n$,
$$ {\mathrm{Per}}_s(F,\omega):={\mathcal{I}}_s(F\omega,F^c\omega)+
{\mathcal{I}}_s(F\omega,F^c\omega^c)+{\mathcal{I}}_s(F\omega^c,F^c\omega),$$
for every~$E\subseteq\omega$ we have that
\begin{equation*} {\mathcal{C}}_{s,\sigma}(E,\omega)={\mathrm{Per}}_s(E,\omega)+
(\sigma-1){\mathcal{I}}_s(E,\omega^c).\end{equation*}
It is also useful to define, for all~$F\subseteq H$ and all~$R>0$,
\begin{equation}\label{PERSSIG}
{\mathrm{Per}}_{s,\sigma}(F,B_R):=
{\mathrm{Per}}_s(F,B_RH) +(\sigma-1){\mathcal{I}}_s(F B_R,H^c)
.\end{equation}
In this setting, we can state the local minimizer condition in~\eqref{MINIMOBLOW}
in terms of the fractional perimeter as follows:

\begin{lemma}\label{EQUILEM}
A set~$E\subseteq H$ is a locally minimizer in~$H$ if and only if, for every~$R>0$,
we have that~${\mathrm{Per}}_s(E,B_RH) <+\infty$ and
$$ {\mathrm{Per}}_{s,\sigma}(E,B_R)
\le {\mathrm{Per}}_{s,\sigma}(F,B_R)
$$
for every~$F\subseteq H$ such that~$F\setminus B_R=E\setminus B_R$.
\end{lemma}

\begin{proof} If~$F\subseteq H$,
\begin{eqnarray*}&&
{\mathrm{Per}}_s(F,B_RH) +(\sigma-1){\mathcal{I}}_s(F B_R,H^c)\\&=&
{\mathcal{I}}_s(F B_RH,F^c B_RH)+
{\mathcal{I}}_s(F B_RH,F^cB_R^cH)+
{\mathcal{I}}_s(F B_RH,F^c H^c)\\&&\quad
+{\mathcal{I}}_s(FB_R^cH,F^c B_RH)+{\mathcal{I}}_s(FH^c,F^c B_RH)
+(\sigma-1){\mathcal{I}}_s(F B_R,H^c)\\
&=&
{\mathcal{I}}_s(F B_R,F^c B_RH)+
{\mathcal{I}}_s(F B_R,F^cB_R^cH)+
{\mathcal{I}}_s(F B_R,H^c)\\&&\quad
+{\mathcal{I}}_s(FB_R^c,F^c B_RH)
+(\sigma-1){\mathcal{I}}_s(F B_R,H^c)\\
&=&
{\mathcal{I}}_s(F B_R,F^c H)
+{\mathcal{I}}_s(FB_R^c,F^c B_RH)
+\sigma {\mathcal{I}}_s(F B_R,H^c).
\end{eqnarray*}
{F}rom this, \eqref{MINIMOBLOW}
and~\eqref{PERSSIG}, the desired result plainly follows.\end{proof}

\section{Extension problems and proof of Proposition~\ref{EXTENSION}}\label{SEXTA}

In this section, we analyze the equivalent extension problem stated in Proposition~\ref{EXTENSION}
and give a proof of it.

\begin{proof}[Proof of Proposition~\ref{EXTENSION}]
First of all, we observe that, by~\eqref{FSSIG} and~\eqref{PERSSIG},
if~$V:\R^{n+1}_+\to\R$ is such that~$V(x,0)=\chi_L(x)$, with~$L\subseteq H$,
and~$\Omega\subset\R^{n+1}$ satisfies~\eqref{TRACCIA},
\begin{equation}\label{CRSTUS}
\begin{split}&
{\mathrm{Per}}_{s,\sigma}(L,B_R)-{\mathcal{F}}_{s,\sigma}(V,\Omega)\\=\;&
{\mathrm{Per}}_s(L,B_RH) +(\sigma-1){\mathcal{I}}_s(L B_R,H^c)\\&\quad
-\int_{\Omega^+} t^{1-s} |\nabla V(X)|^2
\,dX-(\sigma-1)\iint_{B_R\times H^c} \frac{\chi_L(x)}{|x-y|^{n+s}}\,dx\,dy
\\=\;&
{\mathrm{Per}}_s(L,B_RH)-
\int_{\Omega^+} t^{1-s} |\nabla V(X)|^2
\,dX.\end{split}\end{equation}
We also remark that, if~$F\subseteq H$, then
\begin{equation}\label{CONTAZZO}
\begin{split}
&{\mathrm{Per}}_s(F,B_R)-{\mathrm{Per}}_s(F, B_RH)\\=\;&
{\mathcal{I}}_s(FB_R, F^cB_R)+{\mathcal{I}}_s(FB_R, F^cB_R^c)+{\mathcal{I}}_s(FB_R^c, F^cB_R)\\
&\quad-
{\mathcal{I}}_s(FB_RH, F^cB_RH)-{\mathcal{I}}_s(FB_RH, F^c(B_RH)^c)-{\mathcal{I}}_s(F(B_RH)^c, F^cB_RH)
\\=\;&
{\mathcal{I}}_s(FB_R, F^cB_R)+{\mathcal{I}}_s(FB_R, F^cB_R^c)+{\mathcal{I}}_s(FB_R^c, F^cB_R)\\
&\quad-
{\mathcal{I}}_s(FB_R, F^cB_RH)-{\mathcal{I}}_s(FB_R, F^c(B_R^cH\cup H^c))-{\mathcal{I}}_s(F(B_R^cH\cup H^c), F^cB_RH)
\\=\;&
{\mathcal{I}}_s(FB_R, F^cB_RH)+
{\mathcal{I}}_s(FB_R, B_RH^c)\\
&\quad+
{\mathcal{I}}_s(FB_R, F^cB_R^cH)+
{\mathcal{I}}_s(FB_R, B_R^cH^c)
+{\mathcal{I}}_s(FB_R^c, F^cB_RH)
+{\mathcal{I}}_s(FB_R^c, B_RH^c)
\\
&\quad-
{\mathcal{I}}_s(FB_R, F^cB_RH)-{\mathcal{I}}_s(FB_R, F^c B_R^cH)
-{\mathcal{I}}_s(FB_R, H^c)
-{\mathcal{I}}_s(F B_R^c, F^cB_RH)
\\=\;&
{\mathcal{I}}_s(FB_R, B_RH^c)+
{\mathcal{I}}_s(FB_R, B_R^cH^c)
+{\mathcal{I}}_s(FB_R^c, B_RH^c)
-{\mathcal{I}}_s(FB_R, H^c)\\=\;&{\mathcal{I}}_s(FB_R^c, B_RH^c).
\end{split}\end{equation}
We will also exploit Lemma~7.2 of~\cite{MR2675483}, according to which (up to normalizing constants
that we omit), given~$L$, $M$, $\omega\subseteq\R^n$ with~${\mathrm{Per}}_s(L,\omega)$, ${\mathrm{Per}}_s(M,\omega)<+\infty$
and~$L\tilde\omega^c=M\tilde\omega^c$, for~$\tilde\omega\Subset\omega$,
then
\begin{equation}\label{PIV}
\inf \int_{\Omega^+} t^{1-s} \big( |\nabla V(X)|^2-|\nabla\EEEE_M(X)|^2\big)
\,dX={\mathrm{Per}}_s(L,\omega)-{\mathrm{Per}}_s(M,\omega),
\end{equation}
where the infimum is taken among all bounded Lipschitz domains~$\Omega\subseteq\R^{n+1}$
with~$ \Omega\cap\{t=0\}\Subset \omega$ and
among all functions~$V:\R^{n+1}_+\to\R$
such that~$V-\EEEE_M$ is compactly supported in~$\Omega$,
and~$V(x,0)=\chi_L(x)$.

Now, assume that~$E$ is a local minimizer in~$H$, and let~$R$, $\Omega$, $\eta$,
$U$ and~$F$ be as in
the assumptions of Proposition~\ref{EXTENSION}(ii).
In the notation of~\eqref{1.3BIS}, we consider the set
$$ \tilde\Omega:=\left\{ X\in \Omega {\mbox{ s.t. }} {\mathrm{dist}}(X,\partial\Omega)\ge\frac\eta2\right\}
=\Omega\setminus(\partial\Omega)_{\eta/2}.
$$
By the assumptions of Proposition~\ref{EXTENSION}(ii), we know that~$U-\EEEE_E$ is compactly
supported
in~$ \tilde\Omega$. Moreover~$\tilde\Omega\cap\{t=0\}\Subset \Omega\cap\{t=0\}=B_R$.
Therefore, we can exploit~\eqref{PIV} with~$\Omega$ there replaced by~$\tilde\Omega$
and~$\omega$ chosen to be~$B_{R}$, thus obtaining
\begin{eqnarray*}&&  \int_{\Omega^+} t^{1-s} \big( |\nabla U(X)|^2-|\nabla\EEEE_E(X)|^2\big)
\,dX=
\int_{\tilde\Omega^+} t^{1-s} \big( |\nabla U(X)|^2-|\nabla\EEEE_E(X)|^2\big)
\,dX\\&&\qquad\qquad\ge {\mathrm{Per}}_s(F,B_{R})-{\mathrm{Per}}_s(E,B_{R}).\end{eqnarray*}
This and~\eqref{CRSTUS} give that
\begin{eqnarray*}
&&{\mathcal{F}}_{s,\sigma}(\EEEE_E,\Omega)-{\mathcal{F}}_{s,\sigma}(U,\Omega)
\\&=&
{\mathrm{Per}}_{s,\sigma}(E,B_R)-{\mathrm{Per}}_s(E,B_RH)+
\int_{\Omega^+} t^{1-s} |\nabla \EEEE_E(X)|^2\,dX\\&&\quad
-{\mathrm{Per}}_{s,\sigma}(F,B_R)+{\mathrm{Per}}_s(F,B_RH)-
\int_{\Omega^+} t^{1-s} |\nabla U(X)|^2\,dX
\\ &\le& {\mathrm{Per}}_{s,\sigma}(E,B_R)- {\mathrm{Per}}_{s,\sigma}(F,B_R)
\\ &&\qquad+
{\mathrm{Per}}_s(F,B_RH) -{\mathrm{Per}}_s(E,B_RH)
-{\mathrm{Per}}_s(F,B_R)+{\mathrm{Per}}_s(E,B_R).
\end{eqnarray*}
Consequently, recalling~\eqref{CONTAZZO} and the fact that~$E$ and~$F$ coincide
outside~$B_R$,
\begin{eqnarray*}
&&{\mathcal{F}}_{s,\sigma}(\EEEE_E,\Omega)-{\mathcal{F}}_{s,\sigma}(U,\Omega)\\
&\le& {\mathrm{Per}}_{s,\sigma}(E,B_R)- {\mathrm{Per}}_{s,\sigma}(F,B_R)
-{\mathcal{I}}_s(FB_R^c, B_RH^c)+{\mathcal{I}}_s(EB_R^c, B_RH^c)\\
&=&{\mathrm{Per}}_{s,\sigma}(E,B_R)- {\mathrm{Per}}_{s,\sigma}(F,B_R).
\end{eqnarray*}
The locally minimizing property of~$E$ and Lemma~\ref{EQUILEM} thereby
imply that~${\mathcal{F}}_{s,\sigma}(\EEEE_E,\Omega)-{\mathcal{F}}_{s,\sigma}(U,\Omega)\le0$,
that is~\eqref{TRACCIA2}, as desired.

Let us now suppose that, viceversa, the claim in~\eqref{TRACCIA2} holds true.
Our objective is now to check that~$E$ is a local minimizer. To this end,
let~$F\subseteq H$ such that~$F\setminus B_R=E\setminus B_R$.
Also, fixed~$\delta>0$, recalling~\eqref{PIV},
we take a bounded Lipschitz domain~$\Omega^{(\delta)}\subseteq\R^{n+1}$
with~$ \Omega^{(\delta)}\cap\{t=0\}\Subset B_{R+1}$ and
a function~${V^{(\delta)}}:\R^{n+1}_+\to\R$
such that~${V^{(\delta)}}-\EEEE_E$ is compactly supported in~$\Omega^{(\delta)}$,
and~${V^{(\delta)}}(x,0)=\chi_F(x)$, with~$\Omega^{(\delta)}$
and~${V^{(\delta)}}$ attaining the infimum in~\eqref{PIV} with~$\omega:=B_{R+1}$
up to an error~$\delta$, that is
\begin{equation}\label{PiUYA} \int_{(\Omega^{(\delta)})^+} t^{1-s} \big( |\nabla {V^{(\delta)}}(X)|^2-|\nabla\EEEE_E(X)|^2\big)
\,dX-\delta
\le {\mathrm{Per}}_s(F,B_{R+1})-{\mathrm{Per}}_s(E,B_{R+1}).
\end{equation}
Let
$$ \rho':=\sup_{x\in\Omega^{(\delta)}\cap\{t=0\}} |x|\qquad{\mbox{ and }}\qquad
\rho:=\max\{ R,\rho'\}.$$
By construction, we have that~$\rho'\in [0,R+1)$, and thus~$\rho\in[R,R+1)$.
Let also~$\Omega^{(\delta,\rho)}:=\Omega^{(\delta)}\cup {\mathcal{B}}_\rho$.
Then, we have that
\begin{equation}\label{GIUS}
\Omega^{(\delta,\rho)}\cap\{t=0\}=B_\rho.\end{equation}
Furthermore, since~${V^{(\delta)}}=\EEEE_E$ in~$\Omega^{(\delta,\rho)}\setminus
\Omega^{(\delta)}$, we have that
\begin{eqnarray*}&& \int_{(\Omega^{(\delta,\rho)})^+} t^{1-s} \big( |\nabla {V^{(\delta)}}(X)|^2-|\nabla\EEEE_E(X)|^2\big)
\,dX\\&=&
\int_{(\Omega^{(\delta)})^+} t^{1-s} \big( |\nabla {V^{(\delta)}}(X)|^2-|\nabla\EEEE_E(X)|^2\big)
\,dX.\end{eqnarray*}
Therefore, recalling~\eqref{PiUYA},
\begin{equation} \label{FSSIGx6}
\int_{(\Omega^{(\delta,\rho)})^+} t^{1-s} \big( |\nabla {V^{(\delta)}}(X)|^2-|\nabla\EEEE_E(X)|^2\big)
\,dX-\delta
\le {\mathrm{Per}}_s(F,B_{R+1})-{\mathrm{Per}}_s(E,B_{R+1}).
\end{equation}
Moreover, in view of~\eqref{GIUS}, we are in the position of
using~\eqref{TRACCIA2} (with~$\Omega$ replaced by~$\Omega^{(\delta,\rho)}$
and~$R$ replaced by~$\rho$). In this way, we find that
$$ {\mathcal{F}}_{s,\sigma}(\EEEE_E,\Omega^{(\delta,\rho)})\le {\mathcal{F}}_{s,\sigma}(V^{(\delta)},\Omega^{(\delta,\rho)}).$$
Consequently, exploiting~\eqref{FSSIG}, \eqref{GIUS} and~\eqref{FSSIGx6},
\begin{eqnarray*}&&
{\mathrm{Per}}_s(E,B_{R+1})-{\mathrm{Per}}_s(F,B_{R+1})\\&\le&
\int_{(\Omega^{(\delta,\rho)})^+} t^{1-s} \big(|\nabla\EEEE_E(X)|^2- |\nabla {V^{(\delta)}}(X)|^2\big)
\,dX+\delta\\&=&
{\mathcal{F}}_{s,\sigma}(\EEEE_E,\Omega^{(\delta,\rho)})
-{\mathcal{F}}_{s,\sigma}({V^{(\delta)}},\Omega^{(\delta,\rho)})\\&&\quad
-(\sigma-1)\iint_{B_\rho\times H^c} \frac{\chi_E(x)}{|x-y|^{n+s}}\,dx\,dy
+(\sigma-1)\iint_{B_\rho\times H^c} \frac{\chi_F(x)}{|x-y|^{n+s}}\,dx\,dy+\delta\\
&\le&
-(\sigma-1)\left( \iint_{B_R\times H^c} \frac{\chi_E(x)}{|x-y|^{n+s}}\,dx\,dy
-\iint_{B_R\times H^c} \frac{\chi_F(x)}{|x-y|^{n+s}}\,dx\,dy\right)+\delta
\\&=&-
(\sigma-1)\,\Big(
{\mathcal{I}}_s(EB_R,H^c)-{\mathcal{I}}_s(FB_R,H^c)\Big)+\delta.\end{eqnarray*}
Hence, since
\begin{eqnarray*}
&&\big({\mathrm{Per}}_s(E,B_{R+1})-{\mathrm{Per}}_s(F,B_{R+1})\big)-
\big( {\mathrm{Per}}_s(E,B_R)-{\mathrm{Per}}_s(F,B_R)\big)\\
&=& {\mathcal{I}}_s(E B_{R+1}B_R^c,E^cB_{R+1}^c)+
{\mathcal{I}}_s(E B_R^c,E^cB_{R+1}B_R^c)\\&&\qquad-
{\mathcal{I}}_s(F B_{R+1}B_R^c,F^cB_{R+1}^c)-
{\mathcal{I}}_s(F B_R^c,F^cB_{R+1}B_R^c)\\&=&0,
\end{eqnarray*}
we find that
$$ {\mathrm{Per}}_s(E,B_{R})-{\mathrm{Per}}_s(F,B_{R})\le -
(\sigma-1)\,\Big(
{\mathcal{I}}_s(EB_R,H^c)-{\mathcal{I}}_s(FB_R,H^c)\Big)+\delta.$$
Then, by~\eqref{PERSSIG} and~\eqref{CONTAZZO},
\begin{eqnarray*}
&&{\mathrm{Per}}_{s,\sigma}(E,B_R)-{\mathrm{Per}}_{s,\sigma}(F,B_R)\\
&=& {\mathrm{Per}}_s(E,B_R H)-{\mathrm{Per}}_s(F,B_R H)+(\sigma-1)\,\Big(
{\mathcal{I}}_s(EB_R,H^c)-{\mathcal{I}}_s(FB_R,H^c)\Big)
\\&\le&\delta+{\mathrm{Per}}_s(E,B_R H)
-{\mathrm{Per}}_s(E,B_R )+{\mathrm{Per}}_s(F,B_R )
-{\mathrm{Per}}_s(F,B_R H)\\&=&\delta-
{\mathcal{I}}_s(EB_R^c, B_RH^c)+
{\mathcal{I}}_s(FB_R^c, B_RH^c)\\&=&\delta.
\end{eqnarray*}
Sending~$\delta\searrow0$, we thereby conclude that~$ {\mathrm{Per}}_{s,\sigma}(E,B_R)\le{\mathrm{Per}}_{s,\sigma}(F,B_R)$.
This, combined with Lemma~\ref{EQUILEM}, gives that~$E$
is a locally minimizer, as desired.
\end{proof}

\section{Monotonicity formula and proof of Theorem~\ref{MONOTONICITY}}\label{B09j}

Goal of this section is proving Theorem~\ref{MONOTONICITY}.

\begin{proof}[Proof of Theorem~\ref{MONOTONICITY}]
Let
$$ C_E:= \left\{ x\in\R^n\setminus\{0\} {\mbox{ s.t. }} \frac{x}{|x|}\in E \right\}.$$
Given~$\e>0$, we define
$$ E^{(\e)}:= \left(  \big((1-\e)E\big)\cap B_{1-\e}\right)\cup \left( C_E
\cap\big(B_1\setminus B_{1-\e}\big)\right)\cup\left( E\setminus B_1
\right),$$
see
\begin{figure}
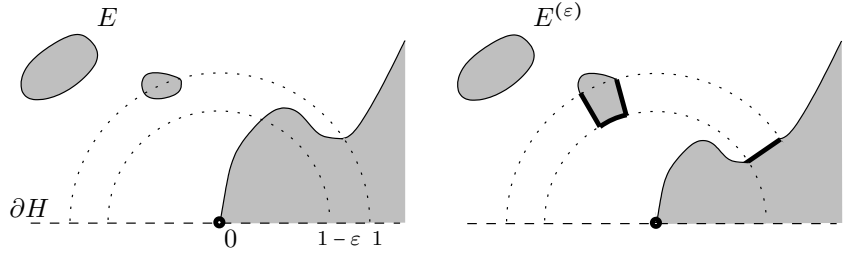\caption{{\small The construction used in the proof of Theorem~\ref{MONOTONICITY}. The parts of the boundary of $E^{(\varepsilon)}$ due to $C_E
\cap\big(B_1\setminus B_{1-\e}\big)$ are depicted by bold lines.}}\label{fig gatto}
\end{figure}
Figure \ref{fig gatto}, and
$$ U_\e(X):=
\begin{cases}
\EEEE_E \left( \frac{X}{1-\e}\right) & {\mbox{ if }} X\in {\mathcal{B}}_{1-\e}^+,\\
\EEEE_E \left( \frac{X}{|X|}\right) & {\mbox{ if }} X\in {\mathcal{B}}_{1}^+\setminus{\mathcal{B}}_{1-\e}^+,\\
\EEEE_E (X) & {\mbox{ if }} X\in \R^{n+1}_+\setminus
{\mathcal{B}}_1.
\end{cases}$$
We remark that
\begin{eqnarray*} U_\e(x,0)&=&
\begin{cases}
\chi_E \left( \frac{x}{1-\e}\right) & {\mbox{ if }} x\in B_{1-\e},\\
\chi_E \left( \frac{x}{|x|}\right) & {\mbox{ if }} x\in B_1\setminus B_{1-\e},\\
\chi_E (x) & {\mbox{ if }} x\in \R^{n}\setminus
B_1,
\end{cases}
\\&=&\chi_{E^{(\e)}}(x)
.\end{eqnarray*}
We also claim that
\begin{equation}\label{PROP:SI1}
E^{(\e)}\subseteq H.
\end{equation}
Indeed, let~$x\in E^{(\e)}$. If~$x\in B_{1-\e}$,
we have that~$x\in(1-\e)E$, and thus~$\frac{x}{1-\e}\in E$.
Since~$E\subseteq H$, we deduce that~$\frac{x_n}{1-\e}\ge0$,
and consequently~$x_n\ge0$, which gives that~$x\in H$ in this case.

If instead~$x\in B_1\setminus B_{1-\e}$, we have that~$x\in C_E$,
and hence~$\frac{x}{|x|}\in E$. In this case, since~$E\subseteq H$,
we find that~$\frac{x_n}{|x|}\ge0$, and again~$x\in H$.
Finally, if~$x\in B_1^c$, we have that~$x\in E\subseteq H$, which completes the proof of~\eqref{PROP:SI1}.

We also observe that~$U_\e=
\EEEE_E$ outside~${\mathcal{B}}_1$.
Then, in view of~\eqref{PROP:SI1}, we can fix~$\eta>0$
and exploit Proposition~\ref{EXTENSION}
with~$\Omega:={\mathcal{B}}_{1+\eta}$, $R:=1+\eta$, $U:=U_\e$ and~$F:=E^{(\e)}$.
In this way, we conclude that
\begin{equation}\label{TERA:1}
\begin{split}
0\,&\le
{\mathcal{F}}_{s,\sigma}(U_\e,{\mathcal{B}}_{1+\eta})
-{\mathcal{F}}_{s,\sigma}(\EEEE_E,{\mathcal{B}}_{1+\eta})\\&=
\int_{{\mathcal{B}}_{1+\eta}^+} t^{1-s} \Big(|\nabla U_\e(X)|^2
- |\nabla \EEEE_E(X)|^2\Big)\,dX\\&\quad+
(\sigma-1)\,\left(\iint_{B_{1+\eta}\times H^c} \frac{\chi_{E^{(\e)}}(x)}{|x-y|^{n+s}}\,dx\,dy
-\iint_{B_{1+\eta}\times H^c} \frac{\chi_{ E }(x)}{|x-y|^{n+s}}\,dx\,dy\right)
\\ &=
\int_{{\mathcal{B}}_{1}^+} t^{1-s} \Big(|\nabla U_\e(X)|^2
- |\nabla \EEEE_E(X)|^2\Big)\,dX\\&\quad+
(\sigma-1)\,\Big(
{\mathcal{I}}_s(B_{1}E^{(\e)}, H^c)-
{\mathcal{I}}_s(B_{1}E, H^c)\Big).
\end{split}\end{equation}
We set
$$ G(r):=r^{s-n}\int_{{\mathcal{B}}_{r}^+} t^{1-s} |\nabla \EEEE_E(X)|^2\,dX,$$
and, using the change of variable~$Y=(y,\theta):=\frac{X}{1-\e}$, we observe that
\begin{eqnarray*}&&
\int_{{\mathcal{B}}_{1}^+} t^{1-s} |\nabla U_\e(X)|^2\,dX
\\&=&\frac{1}{(1-\e)^2}
\int_{{\mathcal{B}}_{1-\e}^+} t^{1-s} \left|
\nabla \EEEE_E\left(\frac{X}{1-\e}\right)\right|^2\,dX\\&&\qquad+
\int_{{\mathcal{B}}_{1}^+\setminus{\mathcal{B}}_{1-\e}^+}
\frac{t^{1-s}}{|X|^2}\left(\left|\nabla \EEEE_E\left(
\frac{X}{|X|}\right)\right|^2 -
\left|\frac{X}{|X|}\cdot \nabla \EEEE_E\left(
\frac{X}{|X|}\right)\right|^2\right)\,dX
\\&=& (1-\e)^{n-s}
\int_{{\mathcal{B}}_{1}^+} \theta^{1-s} |\nabla \EEEE_E(Y)|^2\,dY\\&&\qquad+\e
\int_{ (\partial{\mathcal{B}}_{1})\cap\{t>0\}}
t^{1-s} \left( |\nabla \EEEE_E(X)|^2 -
\left|X\cdot \nabla \EEEE_E(X)\right|^2\right)\,d{\mathcal{H}}^n_X+o(\e)\\
&=& (1-\e)^{n-s}G(1)+\e
\int_{ (\partial{\mathcal{B}}_{1})\cap\{t>0\}}
t^{1-s} |\nabla_\tau \EEEE_E(X)|^2\,d{\mathcal{H}}^n_X+o(\e),
\end{eqnarray*}
where~$\nabla_\tau$ denotes the tangential gradient along~$\partial{\mathcal{B}}_{1}$.

Similarly,
\begin{eqnarray*}&&
\int_{{\mathcal{B}}_{1}^+} t^{1-s} |\nabla \EEEE_E(X)|^2\,dX\\&=&
\int_{{\mathcal{B}}_{1-\e}^+} t^{1-s} |\nabla \EEEE_E(X)|^2\,dX+\e
\int_{(\partial{\mathcal{B}}_{1})\cap\{t>0\}} t^{1-s} |\nabla \EEEE_E(X)|^2\,d{\mathcal{H}}^n_X+o(\e)\\
&=& (1-\e)^{n-s}G(1-\e)
+\e
\int_{(\partial{\mathcal{B}}_{1})\cap\{t>0\}} t^{1-s} |\nabla \EEEE_E(X)|^2\,d{\mathcal{H}}^n_X+o(\e),
\end{eqnarray*}
and accordingly
\begin{equation}\label{09oKA-2}
\begin{split}
&\int_{{\mathcal{B}}_{1}^+} t^{1-s} \Big(|\nabla U_\e(X)|^2
- |\nabla \EEEE_E(X)|^2\Big)\,dX
\\ =\;&
(1-\e)^{n-s}G(1)-(1-\e)^{n-s}G(1-\e)\\&\quad
+\e\left(
\int_{ (\partial{\mathcal{B}}_{1})\cap\{t>0\}}
t^{1-s} |\nabla_\tau \EEEE_E(X)|^2\,d{\mathcal{H}}^n_X
-\int_{(\partial{\mathcal{B}}_{1})\cap\{t>0\}} t^{1-s} |\nabla \EEEE_E(X)|^2\,d{\mathcal{H}}^n_X\right)+o(\e)\\
=\;&
\left( 1-(n-s)\e\right)\,\left(G(1)-G(1-\e)\right)
-\e
\int_{ (\partial{\mathcal{B}}_{1})\cap\{t>0\}}
t^{1-s} |\nabla_\nu \EEEE_E(X)|^2\,d{\mathcal{H}}^n_X
+o(\e),\end{split}
\end{equation}
where~$\nabla_\nu$ denotes the (exterior)
normal gradient along~$\partial{\mathcal{B}}_{1}$.

Furthermore, setting
$$ J(r):=r^{s-n} {\mathcal{I}}_s(B_{r}E, H^c) ,$$
using the substitutions~$\bar x:=\frac{x}{1-\e}$ and~$\bar y:=\frac{y}{1-\e}$,
and noticing that~$C_E\cap(\partial B_1)=E\cap(\partial B_1)$,
we have that
\begin{eqnarray*}
&& {\mathcal{I}}_s(B_{1}E^{(\e)}, H^c)-
{\mathcal{I}}_s(B_{1}E, H^c)\\
&=& {\mathcal{I}}_s\big(B_{1-\e}\big((1-\e)E\big), H^c\big)-
{\mathcal{I}}_s(B_{1-\e}E, H^c)+
{\mathcal{I}}_s(B_{1}B_{1-\e}^cC_E, H^c)-
{\mathcal{I}}_s(B_{1}B_{1-\e}^cE, H^c)\\
&=& \iint_{ B_{1-\e} ((1-\e)E)\times H^c}\frac{dx\,dy}{|x-y|^{n+s}}
-(1-\e)^{n-s}J(1-\e)\\&&\quad+\e\left(
\iint_{(\partial B_1)\times H^c} \frac{\chi_{C_E}(x)\,d{\mathcal{H}}^{n-1}_x\,dy}{|x-y|^{n+s}}-
\iint_{(\partial B_1)\times H^c} \frac{\chi_{ E }(x)\,d{\mathcal{H}}^{n-1}_x\,dy}{|x-y|^{n+s}}\right)+o(\e)
\\ &=& (1-\e)^{n-s}\iint_{ B_{1 } E\times  H^c}\frac{d\bar x\,d\bar y}{|\bar x-\bar y|^{n+s}}
-(1-\e)^{n-s}J(1-\e)+o(\e)\\&=&
(1-\e)^{n-s} \big( J(1)-J(1-\e)\big)+o(\e)\\
&=& \left( 1-(n-s)\e\right)\,\left(J(1)-J(1-\e)\right)+o(\e).
\end{eqnarray*}
Then, plugging this information and~\eqref{09oKA-2} into~\eqref{TERA:1},
and noticing that~$\Phi_E(r)=G(r)+(\sigma-1)J(r)$, we conclude that
\begin{eqnarray*}
0&\le& \left( 1-(n-s)\e\right)\,\left(G(1)-G(1-\e)\right)
-\e
\int_{ (\partial{\mathcal{B}}_{1})\cap\{t>0\}}
t^{1-s} |\nabla_\nu \EEEE_E(X)|^2\,d{\mathcal{H}}^n_X
\\&&\qquad
+(\sigma-1)\,\left( 1-(n-s)\e\right)\,\left(J(1)-J(1-\e)\right)+o(\e)\\&=&
\left( 1-(n-s)\e\right)\,\left(\Phi_E(1)-\Phi_E(1-\e)\right)
-\e
\int_{ (\partial{\mathcal{B}}_{1})\cap\{t>0\}}
t^{1-s} |\nabla_\nu \EEEE_E(X)|^2\,d{\mathcal{H}}^n_X+o(\e)\\
&=& \e\,\Phi_E'(1)
-\e
\int_{ (\partial{\mathcal{B}}_{1})\cap\{t>0\}}
t^{1-s} |\nabla_\nu \EEEE_E(X)|^2\,d{\mathcal{H}}^n_X+o(\e).
\end{eqnarray*}
Therefore, dividing by~$\e$ and sending~$\e\searrow0$, we see that
\begin{equation} \label{ANCO}\Phi_E'(1)\ge
\int_{ (\partial{\mathcal{B}}_{1})\cap\{t>0\}}
t^{1-s} |\nabla_\nu \EEEE_E(X)|^2\,d{\mathcal{H}}^n_X.\end{equation}
On the other hand, in light of~\eqref{SCALE}, we know that
\begin{equation}\label{SCA22} \Phi_{E_\lambda} (r)= \Phi_{E_{\lambda r/\rho}} (\rho) ,\end{equation}
for all~$r$, $\rho$, $\lambda>0$,
and thus, choosing~$\rho:=\lambda r$,
$$ \Phi_{E_{\lambda}} (r)= \Phi_{ E } (\lambda r) .$$
As a consequence, taking~$\lambda:=R$ and~$r:=1+h$,
and~$\lambda:=R$ and~$r:=1$, we see that, for all~$R>0$,
$$ \Phi_{ E }' (R)=\lim_{h\to0} \frac{ \Phi_{ E } (R(1+h))-\Phi_{ E } (R)}{Rh}=
\lim_{h\to0} \frac{ \Phi_{E_{R}} (1+h)-\Phi_{E_{R}} (1)}{Rh}=\frac{\Phi_{E_R}' (1)}{R}.
$$
Combining this and~\eqref{ANCO} (used here on the set~$E_R$), we obtain that
\begin{eqnarray*}\Phi_{ E }' (R)&
\ge&\frac1R
\int_{ (\partial{\mathcal{B}}_{1})\cap\{t>0\}}
t^{1-s} |\nabla_\nu \EEEE_{E_R}(X)|^2\,d{\mathcal{H}}^n_X\\&
=&R
\int_{ (\partial{\mathcal{B}}_{1})\cap\{t>0\}}
t^{1-s} |\nabla_\nu \EEEE_{ E }(RX)|^2\,d{\mathcal{H}}^n_X\\&
=&
R^{s-n}
\int_{ (\partial{\mathcal{B}}_{R})\cap\{t>0\}}
t^{1-s} |\nabla_\nu \EEEE_{ E }(X)|^2\,d{\mathcal{H}}^n_{X},
\end{eqnarray*}
that is~\eqref{MONI}, as desired.

Now, if~$E$ is a cone, from~\eqref{SCALE} we have that~$
\Phi_E (r)= \Phi_{ E } (\rho)$ for any~$r$, $\rho>0$, and therefore~$\Phi_E$ is constant.

Viceversa, if~$\Phi_E$ is constant, we deduce from~\eqref{MONI} that
$$ \int_{ (\partial{\mathcal{B}}_{r})\cap\{t>0\}}
t^{1-s} |\nabla_\nu \EEEE_{ E }(X)|^2\,d{\mathcal{H}}^n_{X}=0$$
for all~$r>0$, and therefore~$X\cdot\nabla \EEEE_{ E }(X)=0$ for all~$X\in\R^{n+1}_+$.
By Euler's Formula, this gives that~$\EEEE_{ E }$ is homogeneous of degree zero,
and consequently, for any~$\tau>0$,
$$ \chi_E(\tau x)=\EEEE_{ E }(\tau x,0)=\EEEE_{ E }(x,0)=\chi_E(x),$$
and hence~$E$ is a cone.
\end{proof}

\section{Homogeneous structure of the blow-up limits and proof of Corollary~\ref{HOMCON}}\label{COPROOF}

In this section, we analyze the structure of the blow-up limit
of local minimizers and we prove Corollary~\ref{HOMCON}.
To this end, we need the forthcoming
auxiliary result which can be seen as the counterpart of Proposition~9.1
in~\cite{MR2675483}
in our setting.

\begin{lemma}\label{LAER-01}
Let~$E\subseteq H$ be a local minimizer in~$H$.
Let~$E_k\subseteq H$ be a sequence of local minimizers in~$H$
and suppose that~$E_k\to E$ in~$L^1_{\mathrm loc}(\R^n)$ as~$k\to+\infty$.

Then,
$$\lim_{k\to+\infty}\Phi_{E_k}(r)=\Phi_E(r)\qquad{\mbox{ for all }}r>0.$$\end{lemma}

\begin{proof} We note that
\begin{equation}\label{du83b} \begin{split}
r^{n-s}\Phi_{E_k} (r)\,&
={\mathcal{F}}_{s,\sigma}(\EEEE_{E_k},{\mathcal{B}}_r)\\
&= \int_{{\mathcal{B}}_r^+} t^{1-s} |\nabla \EEEE_{E_k}(X)|^2\,dX+(\sigma-1)\iint_{(B_rH)\times H^c} \frac{
\chi_{E_k}(x)
}{|x-y|^{n+s}}\,dx\,dy.\end{split}\end{equation}
By the Dominated Convergence Theorem, we have that
\begin{equation}\label{7675-som} \lim_{k\to+\infty}\iint_{(B_rH)\times H^c} \frac{
\chi_{E_k}(x)
}{|x-y|^{n+s}}\,dx\,dy=
\iint_{(B_rH)\times H^c} \frac{
\chi_{ E }(x)
}{|x-y|^{n+s}}\,dx\,dy.\end{equation}
By this and~\eqref{du83b} we see that, to prove the desired result, it suffices to show that
\begin{equation}\label{du83b-2}
\lim_{k\to+\infty}\int_{{\mathcal{B}}_r^+} t^{1-s} |\nabla \EEEE_{E_k}(X)|^2\,dX=
\int_{{\mathcal{B}}_r^+} t^{1-s} |\nabla \EEEE_{ E }(X)|^2\,dX.
\end{equation}
To this end, we use formula~(7.2) in
Proposition~7.1 in~\cite{MR2675483}
and we write that, given~$r$, $\delta>0$,
\begin{equation*}
\begin{split}&\int_{{\mathcal{B}}_r^+} t^{1-s} |\nabla (\EEEE_{E_k}-\EEEE_{ E })(X)|^2\,dX
=\int_{{\mathcal{B}}_r^+} t^{1-s} |\nabla \EEEE_{\chi_{E_k}-\chi_{ E }} (X)|^2\,dX
\\&\qquad
\le C_{r,\delta}\,
\int_{{\mathcal{Q}}_{r,\delta}}\frac{|(\chi_{E_k}-\chi_{ E })(x)-
(\chi_{E_k}-\chi_{ E })(y)|^2
}{|x-y|^{n+s}}
\,dx\,dy,\end{split}
\end{equation*}
for some~$C_{r,\delta}>0$, where
$$ {{\mathcal{Q}}_{r,\delta}}:=
\R^{2n}\setminus (B_{r+\delta}^c\times B_{r+\delta}^c).$$
Consequently, the claim in~\eqref{du83b-2} is established once we show that
\begin{equation}\label{du83b-25}
\lim_{k\to+\infty}
\int_{{\mathcal{Q}}_{r,\delta}}\frac{|(\chi_{E_k}-\chi_{ E })(x)-
(\chi_{E_k}-\chi_{ E })(y)|^2
}{|x-y|^{n+s}}
\,dx\,dy=0.\end{equation}
It is convenient to define
$$ f_k(x,y):=
\frac{ \chi_{E_k}(x)-\chi_{E_k}(y)}{|x-y|^{\frac{n+s}2}}\qquad{\mbox{ and }}\qquad
f(x,y):=
\frac{\chi_{ E }(x)-\chi_{ E }(y)}{|x-y|^{\frac{n+s}2}}.$$
In this way, claim~\eqref{du83b-25}
can be written as
\begin{equation}\label{du83b-26}
\lim_{k\to+\infty} \|f_k-f\|_{L^2({\mathcal{Q}}_{r,\delta})}=0.\end{equation}

We now use~$\bowtie$ as a short notation for~$\frac{\chi_{{\mathcal{Q}}_{r,\delta}}(x,y)\,dx\,dy}{|x-y|^{n+s}}$
and set~$B:=B_{r+\delta}$. We point out that
\begin{equation}\label{06b ewme}
\begin{split}
\frac{\|f_k\|^2_{L^2({\mathcal{Q}}_{r,\delta})}}2\,&=\iint_{E_k \times E_k^c} \bowtie\\&=
\iint_{(E_kB) \times E_k^c} \bowtie
+\iint_{(E_k B^c)\times E_k^c} \bowtie\\&=
\iint_{(E_kB) \times (E_k^cH)} \bowtie+\iint_{(E_kB) \times (E_k^cH^c)} \bowtie\\&\qquad
+\iint_{(E_k B^c)\times (E_k^cH)} \bowtie
+\iint_{(E_k B^c)\times (E_k^cH^c)} \bowtie\\&={\mathcal{I}}_s(E_kB,E_k^cH)+
{\mathcal{I}}_s(E_kB,E_k^cH^c)+
{\mathcal{I}}_s(E_k B^c,E_k^cBH)+
{\mathcal{I}}_s(E_k B^c,E_k^cBH^c)
\end{split}\end{equation}
and therefore
\begin{eqnarray*}
\frac{\|f_k\|^2_{L^2({\mathcal{Q}}_{r,\delta})}}2
&\le&
{\mathcal{I}}_s(E_kB,E_k^cH)+
{\mathcal{I}}_s(E_kB,E_k^cBH^c)+
{\mathcal{I}}_s(E_k B^c,E_k^cBH)+2
{\mathcal{I}}_s(B,B^c)\\&\le&
{\mathcal{I}}_s(E_kB,E_k^cH)+
{\mathcal{I}}_s(E_k B^c,E_k^cBH)+2
{\mathcal{I}}_s(B,B^c)+
{\mathcal{I}}_s(BH,BH^c)\\&=&
{\mathcal{I}}_s(E_kB,E_k^cH)+
{\mathcal{I}}_s(E_k B^c,E_k^cBH)+C_{r,\delta},
\end{eqnarray*}
with~$C_{r,\delta}$ independent of~$k$.
Hence, using
the local minimizing property of~$E_k$ in~\eqref{MINIMOBLOW}, taking~$F_k:=E_kB^c$,
\begin{eqnarray*}
\frac{\|f_k\|^2_{L^2({\mathcal{Q}}_{r,\delta})}}2&\le&
{\mathcal{I}}_s(F_kB,F_k^cH)+
{\mathcal{I}}_s(F_k B^c,F_k^cBH)
+\sigma\big(
{\mathcal{I}}_s(F_k B,H^c)-
{\mathcal{I}}_s(E_k B,H^c)
\big)
+C_{r,\delta}
\\&\le&0+
{\mathcal{I}}_s( B^c,B)
+\sigma\big(
0-
{\mathcal{I}}_s(E_k B,H^c)
\big)
+C_{r,\delta}\\&\le&2C_{r,\delta}.
\end{eqnarray*}
This and Fatou's Lemma yield that
$$ \|f\|^2_{L^2({\mathcal{Q}}_{r,\delta})}\le4C_{r,\delta}.$$
Now we remark that to prove~\eqref{du83b-26}
it suffices to show that
\begin{equation}\label{du83b-27}
\lim_{k\to+\infty} \|f_k\|_{L^2({\mathcal{Q}}_{r,\delta})}=\|f\|_{L^2({\mathcal{Q}}_{r,\delta})}.\end{equation}
Indeed, suppose that~\eqref{du83b-27} holds true
and notice that~$f_k$ converges to~$f$ pointwise.
Let~$\varphi\in C^\infty_0({\mathcal{Q}}_{r,\delta})$
and observe that
$$ |f_k(x,y)\,\varphi(x,y)|\le \frac{ |\varphi(x,y)|}{|x-y|^{\frac{n+s}2}}\in L^1(
{\mathcal{Q}}_{r,\delta}).$$
Hence, by the Dominated Convergence Theorem,
$$ \lim_{k\to+\infty}\int_{{\mathcal{Q}}_{r,\delta}} f_k\varphi
=\int_{{\mathcal{Q}}_{r,\delta}} f\varphi.$$
By density, given~$\e>0$, we can pick~$\varphi_\e\in C^\infty_0({\mathcal{Q}}_{r,\delta})$
such that~$\| \varphi_\e-f\|_{L^2({\mathcal{Q}}_{r,\delta}))}\le\e$. In this way, we find that
\begin{eqnarray*}&&\limsup_{k\to+\infty}
\left| \int_{{\mathcal{Q}}_{r,\delta}} f_k f
-\int_{{\mathcal{Q}}_{r,\delta}} f^2
\right|\\&\le&\limsup_{k\to+\infty}
\left| \int_{{\mathcal{Q}}_{r,\delta}} f_k\varphi_\e
-\int_{{\mathcal{Q}}_{r,\delta}} f^2
\right|+
\int_{{\mathcal{Q}}_{r,\delta}} f_k|f-\varphi_\e|
\\&\le&\left| \int_{{\mathcal{Q}}_{r,\delta}} f\varphi_\e
-\int_{{\mathcal{Q}}_{r,\delta}} f^2
\right|+\limsup_{k\to+\infty}
\| f_k\|_{L^2({\mathcal{Q}}_{r,\delta}))}
\| \varphi_\e-f\|_{L^2({\mathcal{Q}}_{r,\delta}))}\\&\le&
\limsup_{k\to+\infty}\big(\| f\|_{L^2({\mathcal{Q}}_{r,\delta}))}+
\| f_k\|_{L^2({\mathcal{Q}}_{r,\delta}))}\big)
\| \varphi_\e-f\|_{L^2({\mathcal{Q}}_{r,\delta}))}
\\&\le& 4\e\sqrt{C_{r,\delta}}.
\end{eqnarray*}
Hence, since~$\e$ can be taken arbitrarily small,
$$ \lim_{k\to+\infty}\int_{{\mathcal{Q}}_{r,\delta}} f_k f
=\int_{{\mathcal{Q}}_{r,\delta}} f^2.$$
As a result, if~\eqref{du83b-27} holds true,
we obtain that
$$ \lim_{k\to+\infty} \|f_k-f\|_{L^2({\mathcal{Q}}_{r,\delta})}^2=
\lim_{k\to+\infty} \|f_k\|_{L^2({\mathcal{Q}}_{r,\delta})}^2+
\|f\|_{L^2({\mathcal{Q}}_{r,\delta})}^2-2\int_{{\mathcal{Q}}_{r,\delta}} f_k f
=0,$$
that is~\eqref{du83b-26}.

In view of this observation, to complete the proof of Lemma~\ref{LAER-01},
we are left with proving~\eqref{du83b-27}. As a matter of fact,
by Fatou's Lemma, to prove~\eqref{du83b-27} it suffices to check that
\begin{equation}\label{du83b-28}
\limsup_{k\to+\infty} \|f_k\|_{L^2({\mathcal{Q}}_{r,\delta})}\le\|f\|_{L^2({\mathcal{Q}}_{r,\delta})},\end{equation}
and therefore the remaining part of this proof is devoted to show the latter inequality.
To this end,
we let~$D_k$ be the symmetric difference of~$E_k$ and~$E$, and we define
$$ G_k:=(E B)\cup(E_k B^c).$$
The local minimizing property of~$E_k$ as stated in~\eqref{MINIMOBLOW} yields that
\begin{eqnarray*}&&
{\mathcal{I}}_s(E_kB,E_k^c H)+
{\mathcal{I}}_s(E_k B^c ,E_k^c BH)+\sigma{\mathcal{I}}_s(E_kB,H^c)\\
&\le&
{\mathcal{I}}_s(G_kB,G_k^c H)+{\mathcal{I}}_s(G_k B^c ,G_k^c BH)
+\sigma{\mathcal{I}}_s(G_kB,H^c) \\&=&{\mathcal{I}}_s(EB,G_k^c H)+{\mathcal{I}}_s(E_k B^c ,E^c BH)+\sigma
{\mathcal{I}}_s(EB,H^c)\\
&=&{\mathcal{I}}_s(EB,E^c BH)+
{\mathcal{I}}_s(EB,E_k^c B^cH)+{\mathcal{I}}_s(E_k B^c ,E^c BH)+\sigma
{\mathcal{I}}_s(EB,H^c)\\&\le&
{\mathcal{I}}_s(EB,E^c BH)+
{\mathcal{I}}_s(EB,E^c B^cH)+{\mathcal{I}}_s(E B^c ,E^c BH)+\sigma
{\mathcal{I}}_s(EB,H^c)\\&&\quad
+{\mathcal{I}}_s(EB,D_k B^cH)
+{\mathcal{I}}_s(D_k B^c ,E^c BH)\\&\le&
{\mathcal{I}}_s(EB,E^c H)+{\mathcal{I}}_s(E B^c ,E^c BH)+\sigma
{\mathcal{I}}_s(EB,H^c)
+2{\mathcal{I}}_s(B,D_k B^c).
\end{eqnarray*}
By~\cite{MR2675483} (see in particular the proof of Theorem~3.3 there), we know that
$$ \lim_{k\to+\infty}{\mathcal{I}}_s(B,D_k B^c)=0,$$
and accordingly we can write that
\begin{eqnarray*}&&\limsup_{k\to+\infty}
{\mathcal{I}}_s(E_kB,E_k^c H)+
{\mathcal{I}}_s(E_k B^c ,E_k^c BH)+\sigma{\mathcal{I}}_s(E_kB,H^c)
\\&\le&
{\mathcal{I}}_s(EB,E^c H)+{\mathcal{I}}_s(E B^c ,E^c BH)+\sigma
{\mathcal{I}}_s(EB,H^c).\end{eqnarray*}
Hence, recalling~\eqref{7675-som},
\begin{equation}\label{bthes9kNS}\limsup_{k\to+\infty}
{\mathcal{I}}_s(E_kB,E_k^c H)+
{\mathcal{I}}_s(E_k B^c ,E_k^c BH)
\le
{\mathcal{I}}_s(EB,E^c H)+{\mathcal{I}}_s(E B^c ,E^c BH).\end{equation}
Besides, from~\eqref{06b ewme},
\begin{eqnarray*}
\frac{\|f_k\|^2_{L^2({\mathcal{Q}}_{r,\delta})}}2&=&
{\mathcal{I}}_s(E_kB,E_k^cH)+
{\mathcal{I}}_s(E_kB,E_k^cH^c)+
{\mathcal{I}}_s(E_k B^c,E_k^cBH)+
{\mathcal{I}}_s(E_k B^c,E_k^cBH^c),
\end{eqnarray*}
and a similar formula holds true by replacing~$f_k$ by~$f$ and~$E_k$ by~$E$.

In this way, exploiting again the Dominated Convergence Theorem, we deduce that
\begin{eqnarray*}&&
\limsup_{k\to+\infty}
\frac{\|f_k\|^2_{L^2({\mathcal{Q}}_{r,\delta})}-\|f\|^2_{L^2({\mathcal{Q}}_{r,\delta})}}2\\&=&
\limsup_{k\to+\infty}
{\mathcal{I}}_s(E_kB,E_k^cH)+
{\mathcal{I}}_s(E_kB,E_k^cH^c)+
{\mathcal{I}}_s(E_k B^c,E_k^cBH)+
{\mathcal{I}}_s(E_k B^c,E_k^cBH^c)\\&&\quad-
{\mathcal{I}}_s(EB,E^cH)-
{\mathcal{I}}_s(EB,E^cH^c)-
{\mathcal{I}}_s(E B^c,E^cBH)-
{\mathcal{I}}_s(E B^c,E^cBH^c)\\&=&\limsup_{k\to+\infty}
{\mathcal{I}}_s(E_kB,E_k^cH)+
{\mathcal{I}}_s(E_k B^c,E_k^cBH)-
{\mathcal{I}}_s(EB,E^cH)-
{\mathcal{I}}_s(E B^c,E^cBH).
\end{eqnarray*}
{F}rom this and~\eqref{bthes9kNS} we obtain~\eqref{du83b-28}, as desired.
\end{proof}

With this preliminary work, we can now complete the proof of Corollary~\ref{HOMCON} by arguing as follows.

\begin{proof}[Proof of Corollary~\ref{HOMCON}]
The proof is based on a double blow-up procedure, combined
with the monotonicity formula in Theorem~\ref{MONOTONICITY}.

First of all, we consider the sequence of sets~$E_{1/k}$,
with~$k\in\N$. By Theorem~A.2 in~\cite{MR3717439}, up to a subsequence,
we know that~$\chi_{E_{1/k}}$ converges in~$L^1_{\mathrm loc}(\R^n)$
to~$\chi_{E^\star}$
as~$k\to+\infty$, for a suitable~$E^\star$ contained in a half-space~$H^\star$,
with~$E^\star$ locally minimizing in~$H^\star$. Up to a rigid motion,
we can suppose that~$H^\star=H$.

Now we consider the sequence~$E^\star_{1/h}$, with~$h\in\N$.
Using again
Theorem~A.2 in~\cite{MR3717439}, up to a subsequence,
we see that~$\chi_{E^\star_{1/h}}$ converges as~$h\to+\infty$
in~$L^1_{\mathrm loc}(\R^n)$
to~$\chi_{E_0}$, for a suitable~$E_0\subseteq H$ which is locally minimizing
in~$H$. Also, thanks to Lemma~\ref{LAER-01}, we have
that
\begin{equation}\label{7676813}
\lim_{h\to+\infty} \Phi_{ E^\star_{1/h} }(r)=\Phi_{E_0}(r).\end{equation}
Then, Corollary~\ref{HOMCON} will be established once we prove the following
claims:
\begin{equation}\label{SCA01}
{\mbox{$E_0$ is a cone}}\end{equation}
and
\begin{equation}\label{SCA02}
\begin{split}&
{\mbox{there exists an infinitesimal sequence~$r_j>0$
such that}} \\&{\mbox{$\chi_{E_{r_j}}$ converges to~$\chi_{E_0}$
in~$L^1_{\mathrm loc}(\R^n)$ as~$j\to+\infty$.}}
\end{split}\end{equation}
To prove~\eqref{SCA01}, we exploit~\eqref{SCA22}
with~$\lambda:=1/h$ and~$\rho:=\lambda r$,
by writing
$$\Phi_{E^\star_{1/h}} (r) = \Phi_{E^\star} \left(\frac{r}h\right).$$
Hence, in light of~\eqref{7676813},
\begin{equation}\label{7762} \Phi_{E_0}(r)=
\lim_{h\to+\infty} \Phi_{ E^\star_{1/h} }(r)=\lim_{h\to+\infty}\Phi_{E^\star} \left(\frac{r}h\right)=
\lim_{\delta\searrow0}\Phi_{E^\star} (\delta) .\end{equation}
Notice that the latter limit exists, due to the monotonicity
of the function proved in Theorem~\ref{MONOTONICITY}.
Furthermore, the identity in~\eqref{7762}
says that~$\Phi_{E_0}$ is constant and then, by Theorem~\ref{MONOTONICITY},
$E_0$ must necessarily be a cone, which proves~\eqref{SCA01}.

Now we prove~\eqref{SCA02}. For this, let~$R>0$.
By the convergence of~$E^\star_{1/h}$, we know that, given~$\e>0$,
there exists~$h_0(R,\e)\in\N$ such that, for all~$h\ge h_0(R,\e)$,
\begin{equation}\label{98923}
\int_{B_R} |\chi_{E^\star_{1/h}}(x)-\chi_{E_0}(x)|\,dx\le \e.\end{equation}
On the other hand, by the convergence of~$E_{1/k}$, there exists~$k_0(R,h,\e)\in\N$
such that, for all~$k\ge k_0(R,h,\e)$,
$$ \int_{B_{R/h}} |\chi_{E_{1/k}}(x)-\chi_{E^\star}(x)|\,dx\le \frac\e{h^n}.$$
Scaling back, and using~\eqref{SCA01}, this gives that, for all~$k\ge k_0(R,h,\e)$,
$$ \int_{B_{R}} |\chi_{E_{1/(hk)}}(x)-\chi_{E^\star_{1/h}}(x)|\,dx\le \e.$$
Combining this with~\eqref{98923}, we find that, for all~$k\ge k_\star(R,\e):=
k_0\big(R,h_0(R,\e),\e\big)$,
\begin{eqnarray*}&& \int_{B_{R}} |\chi_{E_{1/(h_0(R,\e)k)}}(x)-\chi_{E_0}(x)|\,dx\\&\le&
\int_{B_{R}} |\chi_{E_{1/(h_0(R,\e)k)}}(x)-\chi_{E^\star_{1/h_0(R,\e)}}(x)|\,dx+
\int_{B_{R}} |\chi_{E^\star_{1/h_0(R,\e)}}(x)-\chi_{E_0}(x)|\,dx\\&\le&2\e.
\end{eqnarray*}
This establishes~\eqref{SCA02}, as desired.
\end{proof}

\section{Locally minimizing cones in the plane and proof of Theorem~\ref{2CONE}}\label{coniin2}

In this section, we take~$n=2$, and we classify
locally minimizing cones, thus proving Theorem~\ref{2CONE}.

\begin{figure}
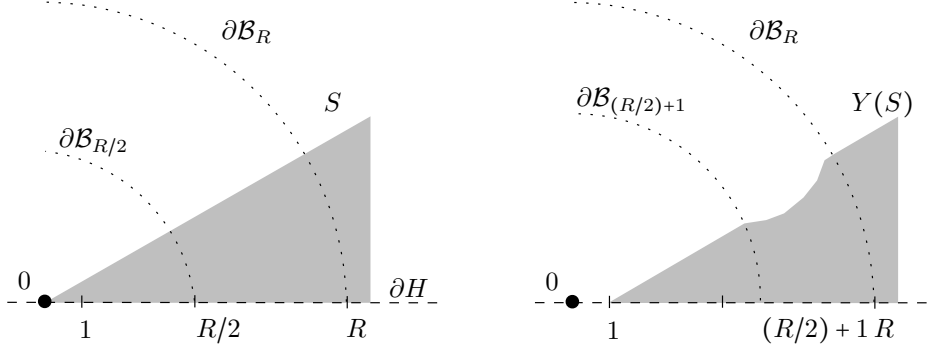\caption{{\small Depicting the action of the map $Y$ defined in \eqref{ANG} on a set $S$. Notice that $S\cap {\mathcal{B}}_{R/2}$ is translated by $e_1$, while $S\setminus{\mathcal{B}}_R$ is left unchanged. Since $\Psi$ is radially decreasing, the slices $S\cap\partial{\mathcal{B}}_\rho$ corresponding to $\rho\in(1,R)$ are translated by multiples $\lambda(\rho)\,e_1$ of $e_1$, where $\lambda(\rho)$ decreases from $\lambda=1$ when $\rho=R/2$, to $\lambda=0$ when $\rho\ge (9/10) R$.}}\label{fig:my-label}
\end{figure}

\begin{proof}[Proof of Theorem~\ref{2CONE}]
Let~$\Psi\in C^\infty_0({\mathcal{B}}_{9/10},\,[0,1])$ be a radially decreasing function with with~$\Psi(X)=1$ for all~$X\in
{\mathcal{B}}_{1/2}$.
Given~$R>2$, to be taken as large as we wish in the following, we consider the transformation
\begin{equation}\label{ANG} \R^3\ni X\mapsto Y:=X+\Psi\left(\frac{X}{R}\right) e_1,\end{equation}
where~$e_1:=(1,0,0)$.
Denoting this map by~$Y(X)$ (see Figure~\ref{fig:my-label}), we see that it is invertible, and
we denote its inverse by~$X(Y)$. We also let
\begin{equation}\label{DISUA}
U:=\EEEE_E,\end{equation} and
$$ U_R^+(Y):= U(X(Y)).$$
We also denote~$U^-_R$ a similar function, in which~$\Psi$
is replaced by~$-\Psi$. In addition, we set~$u(x):=U(x,0)$,
$u^+_R(y):=U^+_R(y,0)$ and~$u^-_R(y):=U^-_R(y,0)$.

We use coordinates~$X=(X_1,X_2,X_3)=(x,t)\in\R^2\times(0,+\infty)$.
We remark that~$Y_3(X)=X_3$, hence~$X_3(Y)=Y_3$,
and accordingly~$X_3(y,0)=0$. This gives that
\begin{equation}\label{FAN}
u^+_R(y)=U(X(y,0))=U(x(y,0),0)=\chi_E(x(y,0)).
\end{equation}

Then, in the notation of~\eqref{FSSIG}, we claim that
\begin{equation}\label{SPA:901}
\big| {\mathcal{F}}_{s,\sigma}(U^+_R,{\mathcal{B}}_R)+
{\mathcal{F}}_{s,\sigma}(U^-_R,{\mathcal{B}}_R)-2
{\mathcal{F}}_{s,\sigma}(U,{\mathcal{B}}_R)\big|\le\frac{C}{R^s},\end{equation}
for some~$C>0$. To prove this, we let
$$ {\mathcal{J}}_R (U):=\int_{{\mathcal{B}}_R^+} t^{1-s} |\nabla U(X)|^2
\,dX\qquad{\mbox{ and }}\qquad
{\mathcal{T}}_R(u):=
\iint_{B_R\times H^c} \frac{u(x)}{|x-z|^{2+s}}\,dx\,dz.$$
A direct computation (see Lemma~1 in~\cite{MR3090533})
shows that
\begin{equation}\label{SPA:902}
\big| {\mathcal{J}}_R(U^+_R)+
{\mathcal{J}}_R(U^-_R)-2
{\mathcal{J}}_R(U)\big|\le\frac{C}{R^s},
\end{equation}
for some~$C>0$.

We introduce the following notation: from now on, we denote
by~$\diamondsuit$ any quantity or bounded function, possibly different from line to line,
which changes sign if~$\Psi$ is replaced by~$-\Psi$. We stress that
it is not necessary that~$\diamondsuit$ has a sign itself,
what matters in this notation is that its pointwise value changes
sign if~$\Psi$ is replaced by~$-\Psi$.

Now, we want to use the change of variable~$\tilde y:=x(y,0)$
and~$\tilde z:=x(y,0)-y+z$. In this way, we have that
$$ \tilde y-\tilde z=y-z.$$
We also observe that, if~$z\in H^c$,
then~$\tilde z_2=x_2(y,0)-y_2+z_2=z_2\le0$, and thus~$\tilde z\in H^c$.

Furthermore, for all~$i$, $j\in\{1,2,3\}$,
$$ D_{X_i} Y_j(X)={\mathrm{Id}}+\frac{\delta_{1j}}{R}\,
\partial_i\Psi\left(\frac{X}{R}\right)={\mathrm{Id}}+\frac{\diamondsuit}{R} .$$
Therefore, we can write that
$$ dy\,dz=\left( 1+\frac{\diamondsuit}{R}+O\left(\frac1{R^2}\right)\right)\,d\tilde y\,d\tilde z.$$
We also point out that
\begin{equation}\label{0-0-he}
{\mbox{if~$y\in B_R$, then~$x(y,0)\in B_R$.}}
\end{equation}
Indeed, if~$|y|\le\frac{99 \,R}{100}$, then
$$ |x(y,0)|=\left|y-\Psi\left(\frac{x(y,0)}{R}\right)e_1\right|
\le \frac{99 \,R}{100}+1<R,$$
as long as~$R$ is large enough.

If instead~$|y|>\frac{99 \,R}{100}$, it follows that
$$ |x(y,0)|=\left|y-\Psi\left(\frac{x(y,0)}{R}\right)e_1\right|
\ge|y|-1>\frac{99 \,R}{100}-1>\frac{9\,R}{10},$$
and consequently~$\Psi\left(\frac{x(y,0)}{R}\right)=0$,
whence~$x(y,0)=y$ in this case.

These considerations prove~\eqref{0-0-he}.
Hence, recalling~\eqref{FAN},
\begin{eqnarray*}
{\mathcal{T}}_R(u^+_R)&=&
\iint_{B_R\times H^c} \frac{u_R^+(y)}{|y-z|^{2+s}}\,dy\,dz\\&=&
\iint_{B_R\times H^c} \frac{\chi_E(x(y,0))}{|y-z|^{2+s}}\,dy\,dz
\\&=&
\iint_{B_R\times H^c} \frac{\chi_E(\tilde y)}{|\tilde y-\tilde z|^{2+s}}
\left( 1+\frac{\diamondsuit}{R}+O\left(\frac1{R^2}\right)\right)
\,d\tilde y\,d\tilde z.
\end{eqnarray*}
Given our notation related to~$\diamondsuit$, this also says that
$$ {\mathcal{T}}_R(u^-_R)=\iint_{B_R\times H^c} \frac{\chi_E(\tilde y)}{|\tilde y-\tilde z|^{2+s}}
\left( 1-\frac{\diamondsuit}{R}+O\left(\frac1{R^2}\right)\right)
\,d\tilde y\,d\tilde z.$$
As a consequence,
\begin{eqnarray*}&& \big|{\mathcal{T}}_R(u^+_R)+ {\mathcal{T}}_R(u^-_R)-2
{\mathcal{T}}_R(u)\big|\le
O\left(\frac1{R^2}\right)
\iint_{B_R\times H^c} \frac{\chi_E(\tilde y)}{|\tilde y-\tilde z|^{2+s}}
\,d\tilde y\,d\tilde z\\&&\qquad\qquad
\le O\left(\frac1{R^2}\right)
\iint_{B_R H\times H^c} \frac{d\tilde y\,d\tilde z}{|\tilde y-\tilde z|^{2+s}}
\le O\left(\frac1{R^2}\right){\mathcal{I}}_s( B_RH, (B_RH)^c)
=O\left(\frac1{R^s}\right).
\end{eqnarray*}
{F}rom this, \eqref{FSSIG} and~\eqref{SPA:902}, we obtain~\eqref{SPA:901},
up to renaming~$C>0$,
as desired.

Moreover, from~\eqref{TRACCIA2}, we can write that
$$ {\mathcal{F}}_{s,\sigma}(U,{\mathcal{B}}_R)
\le {\mathcal{F}}_{s,\sigma}(U_R^-,{\mathcal{B}}_R).$$
Using this and~\eqref{SPA:901}, we conclude that
\begin{equation}\label{6.5}
{\mathcal{F}}_{s,\sigma}(U^+_R,{\mathcal{B}}_R)-
{\mathcal{F}}_{s,\sigma}(U,{\mathcal{B}}_R)\le {\mathcal{F}}_{s,\sigma}(U^+_R,{\mathcal{B}}_R)+
{\mathcal{F}}_{s,\sigma}(U^-_R,{\mathcal{B}}_R)-2
{\mathcal{F}}_{s,\sigma}(U,{\mathcal{B}}_R)\le\frac{C}{R^s}.
\end{equation}
Now we claim that
\begin{equation}\label{ALAM}
\begin{split}&
{\mbox{$U$ is monotone in the direction~$e_1$,}}\\&{\mbox{namely
either~$U(X+\tau e_1)\ge U(X)$ or~$U(X+\tau e_1)\le U(X)$, for every~$\tau>0$.}}
\end{split}\end{equation}
To prove this, we argue by contradiction, supposing that
there exist~$\bar X\in\R^3_+$ and~$\bar\tau_1$, $\bar\tau_2>0$ such that
\begin{equation}\label{ALAM2a}
U(\bar X+\bar\tau_1 e_1)> U(\bar X)\qquad{\mbox{ and }}\qquad
U(\bar X+\bar\tau_2 e_1)< U(\bar X).
\end{equation}
Since~$E$ is a cone, we have that~$U$ is homogeneous of degree zero,
and therefore, letting~$P:=\bar\tau_1^{-1}\bar X$
and~$Q:=\bar\tau_2^{-1}\bar X$, we can write~\eqref{ALAM2a} as
\begin{equation}\label{ALAM2}\begin{split}&
U(P+e_1)=U(\bar\tau_1^{-1}\bar X+e_1)=
U(\bar X+\bar\tau_1e_1)> U(\bar X)=U(\bar\tau_1^{-1}\bar X)=
U(P)
\\{\mbox{ and }}\qquad&
U(Q+e_1)=U(\bar\tau_2^{-1}\bar X+e_1)=
U(\bar X+\bar\tau_2 e_1)< U(\bar X)=
U(\bar\tau_2^{-1}\bar X)=U(Q).\end{split}
\end{equation}
We can suppose that
\begin{equation}\label{LARGO}
R/2>M:=2+|Q|+|P|,\end{equation} and we set
$$ V_R(X):=\min\{ U(X),\,U^+_R(X)\}\qquad{\mbox{ and }}\qquad
W_R(X):=\max\{ U(X),\,U^+_R(X)\}.$$
We remark that
\begin{equation}\label{SH:923e2}
{\mathcal{F}}_{s,\sigma}(V_R,{\mathcal{B}}_R)+
{\mathcal{F}}_{s,\sigma}(W_R,{\mathcal{B}}_R)=
{\mathcal{F}}_{s,\sigma}(U,{\mathcal{B}}_R)+
{\mathcal{F}}_{s,\sigma}(U^+_R,{\mathcal{B}}_R).
\end{equation}
In addition, by~\eqref{TRACCIA2},
$$ {\mathcal{F}}_{s,\sigma}(U,{\mathcal{B}}_R)\le
{\mathcal{F}}_{s,\sigma}(V_R,{\mathcal{B}}_R).$$
Combining this and~\eqref{SH:923e2}, we find that
\begin{equation}\label{SH:923e3}
{\mathcal{F}}_{s,\sigma}(W_R,{\mathcal{B}}_R)\le
{\mathcal{F}}_{s,\sigma}(U^+_R,{\mathcal{B}}_R).
\end{equation}
Now, we denote by~$W_\star$ the minimizer of~${\mathcal{J}}_M(W)$
among all the competitors~$W$ with~$W=W_R$
on~$\partial{\mathcal{B}}_M^+=\big(
(\partial{\mathcal{B}}_M)\cap\{t>0\}\big)\cup\big( B_M\times\{0\}\big)$.

We remark that the minimization of the functional leads to the equation
\begin{equation}\label{41} {\mathrm{div}}\,(t^{1-s}\nabla W_\star)=0\quad{\mbox{ in }}\,
{\mathcal{B}}_M^+.\end{equation}
Also, the same equation is fulfilled by~$U$, in view of~\eqref{DISUA}.

We claim that
\begin{equation}\label{NONAUG}
W_\star\ne W_R.
\end{equation}
Indeed, suppose by contradiction
that~$W_\star= W_R$. Then, since~$U\le W_R=W_\star$,
we deduce by the Strong Maximum Principle for the equation in~\eqref{41}
(see e.g. Corollary~2.3.10 in~\cite{MR643158}) that
\begin{equation}\label{EIT}
{\mbox{either~$U<W_R$
or~$U=W_R$ in~${\mathcal{B}}_M^+$.}}\end{equation}
On the other hand, by~\eqref{LARGO}, we have that, for all~$i\in\{1,2\}$,
$$ Y(P)=P+\Psi\left(\frac{P}{R}\right)e_1=P+ e_1\qquad{\mbox{and}}\qquad
Y(Q)=Q+\Psi\left(\frac{Q}{R}\right)e_1=Q+ e_1.$$
Consequently, by~\eqref{ALAM2},
\begin{eqnarray*}&&
U_R^+( Y(P))=U(P)<U(P+e_1)=U(Y(P))\\{\mbox{and }}&&
U_R^+( Y(Q))=U(Q)>U(Q+e_1)=U(Y(Q)).\end{eqnarray*}
Therefore, we see that~$W_R(Y(P))=U(Y(P))$
and~$W_R(Y(Q))=U^+_R(Y(Q))>U(Y(Q))$,
and these observations say that none of the two possibilities in~\eqref{EIT}
can be fulfilled.

This contradiction proves~\eqref{NONAUG}. Then, from~\eqref{NONAUG},
we obtain that there exists~$\delta_0>0$ such that
$$ {\mathcal{J}}_M(W_\star)+\delta_0\le
{\mathcal{J}}_M(W_R).$$
We stress that this~$\delta_0$ is independent of~$R$, because~$W_R$
in~${\mathcal{B}}_M$ does not depend on~$R$, being
$$ W_R(X)=\max\{U(X), \,U(X-e_1)\}\qquad{\mbox{
for all }}X\in{\mathcal{B}}_M^+,$$
thanks to~\eqref{LARGO}.

Furthermore, if we extend~$W_\star$ to be equal to~$W_R$
outside~${\mathcal{B}}_M^+$, we have that
\begin{equation}\label{TRAga}
{\mathcal{J}}_R(W_R)-{\mathcal{J}}_R(W_\star)=
{\mathcal{J}}_M(W_R)-{\mathcal{J}}_M(W_\star)\ge\delta_0.
\end{equation}
Since, by construction~$w_\star(x):=W_\star(x,0)=W_R(x,0)=:w_R(x)$,
we have that~${\mathcal{T}}_R(w_\star)={\mathcal{T}}_R(w_R)$.
This and~\eqref{TRAga} give that
$$ {\mathcal{F}}_{s,\sigma}(W_R,{\mathcal{B}}_R)-
{\mathcal{F}}_{s,\sigma}(W_\star,{\mathcal{B}}_R)\ge\delta_0.$$
As a consequence, in light of~\eqref{SH:923e3},
\begin{equation}\label{0oA73} {\mathcal{F}}_{s,\sigma}(U^+_R,{\mathcal{B}}_R)-
{\mathcal{F}}_{s,\sigma}(W_\star,{\mathcal{B}}_R)\ge\delta_0.\end{equation}
On the other hand, using again~\eqref{TRACCIA2},
$$ {\mathcal{F}}_{s,\sigma}(U,{\mathcal{B}}_R)\le
{\mathcal{F}}_{s,\sigma}(W_\star,{\mathcal{B}}_R).$$
Comparing this and~\eqref{0oA73}, we see that
\begin{equation*} {\mathcal{F}}_{s,\sigma}(U^+_R,{\mathcal{B}}_R)-
{\mathcal{F}}_{s,\sigma}(U,{\mathcal{B}}_R)\ge\delta_0.\end{equation*}
Hence, recalling~\eqref{6.5},
$$ \frac{C}{R^s}\ge\delta_0.$$
We can now send~$R\to+\infty$ and find that~$0\ge\delta_0>0$.
This contradiction proves the validity of~\eqref{ALAM}.

As a consequence of~\eqref{ALAM}, we have that~$u$
is monotone in the direction~$e_1$, hence the cone~$E$
is made of only one component.

{F}rom this and Theorem~1.4 in~\cite{MR3717439}, one also
obtains~\eqref{ANGO}.\end{proof}

\vfill

\end{document}